\documentclass[a4paper,10pt]{amsart}
\usepackage{hyperref}
 \usepackage{amsfonts,mathrsfs}
 \usepackage{amsthm}
 \usepackage{amssymb}
 \usepackage{enumerate}

 \usepackage{mathtools}
\theoremstyle{definition}
\newtheorem{Def}{Definition}[section]
\newtheorem{Bsp}[Def]{Example}
\newtheorem{Bem}[Def]{Remark}
\theoremstyle{plain}
\newtheorem{Prop}[Def]{Proposition}
\newtheorem{Thm}[Def]{Theorem}
\newtheorem*{Thm*}{Theorem}
\newtheorem{Lem}[Def]{Lemma}
\newtheorem{Kor}[Def]{Corollary}
\newtheorem*{Kor*}{Corollary}

\newtheorem*{con*}{Conjecture}
\newtheorem{frag}{Question}
\newtheorem*{verm*}{Vermutung}

\usepackage{mathrsfs}

\newcommand{\Supp}{\operatorname{Supp}}

\newcommand{\lk}{\operatorname{lk}} 

\newcommand{\cL}{{\mathcal L}}

\newcommand{\cO}{{\mathcal O}}

\newcommand{\C}{{\mathbb C}}
\newcommand{\R}{{\mathbb R}}
\newcommand{\pp}{\mathbb{P}}

\newcommand{\N}{{\mathbb N}}
\newcommand{\Z}{{\mathbb Z}}

\title{The separating semigroup of a real curve}
\author{Mario Kummer}
\author{Kristin Shaw}
\address{Technische Universit\"at Berlin,
Institut f\"ur Mathematik, 
Stra{\ss}e des 17.~Juni 136
10623 Berlin, Germany} 
\email{kummer@tu-berlin.de}
\address{Department of Mathematics, University of Oslo, P.O.~Box 1053,   Blindern, 0316 Oslo, Norway}
\email{krisshaw@math.uio.no}
\thanks{The second author's research was supported in part by a postdoctoral fellowship from the  Alexander von Humboldt foundation. The project {was} completed while both authors were at the Max-Planck-Institute for Mathematics in the Sciences, Leipzig. We thank the institute for the excellent working conditions. }

\usepackage{xcolor}

\newcommand{\comment}[1]{}

\begin{document}


\begin{abstract}

 We introduce the separating  semigroup of a real algebraic curve of dividing type. The elements of this semigroup  record the  possible  degrees of the covering maps obtained by restricting separating morphisms to the real part of the curve. 
We also introduce the hyperbolic semigroup which consists of elements of the separating semigroup 
arising from morphisms which are  compositions of a linear projection with an embedding of the curve to some projective space.

We completely determine both semigroups in the case of maximal curves.  We also prove that any embedding of a real curve of dividing type to projective space of sufficiently high degree is hyperbolic. 
Using these semigroups we show that the hyperbolicity locus of an embedded curve is in general not connected. 

\vspace{0.3cm}
\noindent
{\sc Resum\'e.}
Nous introduisons le semi-groupe s\'eparant d'une courbe alg\'ebrique r\'eelle s\'eparante. Les \'el\'ements de ce semi-groupe gardent trace des degr\'es possibles des rev\^etements du cercle obtenus par restriction \`a la partie r\'eelle de la courbe
 des morphismes s\'eparants. Nous introduisons aussi le semi-groupe hyperbolique, compos\'e des \'el\'ements  du semi-groupe s\'eparant provenant des morphismes qui sont la composition d'une projection lin\'eaire et d'un plongement de la courbe dans un espace projectif. 

Nous d\'eterminons les deux groupes dans le cas des courbes maximales. Nous d\'emontrons aussi que tout plongement d'une courbe r\'eelle s\'eparante de degr\'e suffisamment grand est hyperbolique. En utilisant ces semi-groupes, nous montrons que le lieu hyperbolique d'une courbe plong\'ee n'est en g\'en\'eral pas connexe.
\end{abstract}
\maketitle

\section{Introduction}
Here a curve will always be a non-singular projective and geometrically irreducible algebraic curve over $\R$. Furthermore, we always use $\pp^n$ to denote  
the projective space defined over $\R$. For a variety $V$ defined over $\R$ we denote by $V(\R)$ and $V(\C)$ the real and complex points of $V$, respectively.

A basic fact concerning the classification of real algebraic curves, or real Riemann surfaces,  is the following dichotomy which goes back to Klein \cite[\S 23]{Kl23}: 
If $X$ is a curve, then the set $X(\C)\smallsetminus X(\R)$ has either one or two connected components. If the latter is the case, $X$ is called \textit{of dividing type}. Curves of dividing type are often called \textit{curves of type I} or \textit{separating} in the literature. If there exists a morphism $f \colon X\to\pp^1$ with the property $f^{-1}(\pp^1(\R))=X(\R)$, then $X(\C)\smallsetminus X(\R)$ can not be connected. This is since $\pp^1(\C)\smallsetminus \pp^1(\R)$ has two connected components and their preimages under $f$ yield two connected components of $X(\C)\smallsetminus X(\R)$. Therefore, such a morphism is also called \textit{separating} since it certifies that $X$ is of dividing type. It follows from the work of {Ahlfors} \cite[\S 4.2]{Ahl50}, though proved in a different context, 
 that conversely every separating curve admits a separating morphism. 

{Rokhlin} \cite{rohlin} used the existence of separating morphisms given by pencils as a certification of certain real plane curves being of dividing type.
In Mikhalkin's \cite{Mik:amoebas}  study of extremal topology of real curves in $(\C^*)^2$, he showed that 
the logarithmic Gauss map of a simple Harnack curve is separating. Conversely, it was shown by Passare and Risler  {\cite{RislerPassare}} that if a planar curve has separating logarithmic Gauss map, then it is a Harnack curve.

The existence of separating morphisms and their properties have  been considered by several authors \cite{huisman, gabard, cophuis, cop}. For example, Gabard \cite{gabard} showed that every separating curve $X$ admits a separating morphism of degree at most $\frac{g+r+1}{2}$ where $g$ is the genus and $r$ the number of connected components of $X(\R)$. Later Coppens \cite{cop} constructed, for every value of  $k$ between $r$ and $\frac{g+r+1}{2}$,  a separating curve $X$ of genus $g$ with $X(\R)$ having $r$ components such that $k$ is the smallest possible degree of a separating morphism. 

In this work, we take a complementary approach. Namely, we fix a curve $X$ of dividing type of genus $g$ and study the set of all separating morphisms $X\to\pp^1$. Let $X(\R)$ consist of $r$ connected components $X_1,\ldots,X_r$. Since $X$ is of dividing type $r+g$ must be odd.
A separating morphism $f$ is always unramified over $X(\R)$ \cite[Theorem 2.19]{kummer2015real}. Therefore, the restriction of $f$ to each $X_i$ is a covering map of $\pp^1(\R)$.   
{This implies that the  degree of a separating morphism is  at least $r$.}
Let  $\N$  denote the positive integers, so $0 \not \in \N$. 
We denote by $d_i(f) \in \mathbb{N}$ the degree of the covering map  $X_i\to\pp^1$ and set $d(f):=(d_1(f),\ldots,d_r(f)) { \in \mathbb{N}^r}$.  
For  $d = (d_1, \dots, d_r)  \in \N^r$ we let $|d| := \sum_{i = 1}^r d_i$. 
Our first main object of interest is the set of all such  
 degree partitions.

\begin{Def}
 The set \[\textnormal{Sep}(X)=\{d(f)\in\N^r \ | \, f \colon X\to\pp^1 \textrm{ separating}\}\] is called the \textit{separating semigroup}.

\end{Def}
Since we assume that 
 $X$ is a separating curve,  the set  $\textnormal{Sep}(X)$ is always non-empty. 
The term semigroup is justified by the fact that this set  turns out to be closed under componentwise addition, see Proposition \ref{prop:semigroup}. 

In Corollary \ref{cor:orthants}, we show that $d+\Z_{\geq0}^{r}\subset \textnormal{Sep}(X)$ for every $d\in \textnormal{Sep}(X)$  with $|d|$ sufficiently large. Our main technique is making use of \textit{interlacing sections}: Two sections $s$ and $s'$ of a line bundle on $X$ are called interlacing if they both have only simple and real zeros and if on each component $X_i$ between each two consecutive zeros of $s$ there is exactly one zero of $s'$. This notion generalizes the notion of interlacing polynomials to sections of {line bundles on} algebraic curves. The concept of interlacing polynomials has attracted a lot of attention since Marcus, Spielman and Srivastava used it to solve the Kadison--Singer problem as well as to find bipartite Ramanujan graphs of all degrees \cite{interl1, interl2}. We will make use of the fact that the morphism to $\pp^1$ defined by $s$ and $s'$ is separating if and only if $s$ and $s'$ are interlacing. This is proved in Lemma \ref{lem:interl}.

We also study the subset $\textnormal{Sep}(X)$ consisting of all degree partitions that are realized by a separating morphism which is actually a linear projection of some embedding of $X$ in projective space from a linear space disjoint from $X$. This is motivated by the following definition from \cite{Sha14}.

\begin{Def}
 Let $X\subset\pp^n$ {be a curve} and $E\subset\pp^n$ be a linear subspace of codimension two such that  $E\cap X=\emptyset$. Then $X$ is \textit{hyperbolic with respect to} $E$ if the linear projection $\pi_E:X\to\pp^1$ from $E$ is separating.  
\end{Def}

Following the terminology of \cite{Sha14} we will call such embedded curves \textit{hyperbolic}.  These curves are a generalization of planar curves defined by hyperbolic polynomials {in three homogeneous variables}. In general hyperbolic polynomials have attracted interest in different areas of mathematics like partial differential equations \cite{pde1, pde2}, optimization \cite{optim1, optim2} and combinatorics \cite{combin1, comb2}. 

Whereas at first sight hyperbolicity might seem to be a rare phenomenon, it is actually quite ubiquitous in the case of curves as justified by the next theorem. It  says that  for a given separating curve $X$, \textit{every} embedding of high enough degree turns out to be hyperbolic. We first remark any divisor $D$ on a curve $X$ with degree $k > 2g$ is very ample by \cite[Corollary 3.2 b)]{Hart77}, and therefore the map $X  \to \pp(\mathscr{L}(D)^\vee)$ is an embedding.

\begin{Thm}\label{thm:alwayshyperbolic}
 Let $X$ be a curve of dividing type of genus $g$. There exists a $k>2g$ with the following property: For any divisor $D$ of degree at least $k$ the corresponding embedding of $X$ to $\pp(\mathscr{L}(D)^\vee)$ is hyperbolic.
\end{Thm}

\begin{Def}\label{def:hyp}
 The \textit{hyperbolic  semigroup}
 $\textnormal{Hyp}(X)$ is the set of all elements of $\textnormal{Sep}(X)$ where the corresponding $f$ can be chosen to be the composition of a linear projection with an embedding of $X$ to some $\pp^n$, where the center of the projection is  disjoint from $X$. 
 \end{Def}

\begin{Bem}
 Replacing $\pp^n$ by $\pp^3$ in Definition \ref{def:hyp}  results in an equivalent condition \cite[\S 2]{kummer2015real}. Also, in the definition of $\textnormal{Hyp}(X)$, one could equivalently just  require  $f$ to be separating and $f^*\cO_{\pp^1}(1)$ to be very ample.
\end{Bem}

The set $\textnormal{Hyp}(X)$  also turns out to be a semigroup, see Proposition \ref{prop:semigroup}.
In Proposition \ref{prop:topo}, we  give an equivalent criterion for a curve to be hyperbolic
in terms of the linking numbers  of its components with the linear subspace from which we project. 

For a planar curve  $X$ of dividing type, a pencil of curves is said to be totally real with respect to $X$ if every curve in the pencil intersects $X$ in only real points. 
In \cite{touze2013totally},  Fiedler-Le Touz{\'e} asks  if for every planar curve of dividing type there exists a totally real pencil. 
Using our techniques we can answer this question in the affirmative even when the base points of the pencil are not contained in the curve. 
 In the next theorem, we let $\mathcal{V}$   denote the {subvariety} of $\mathbb{P}^2$ defined by a collection of homogeneous polynomials in $\R[x,y,z]$.

\begin{Thm}\label{thm:pencils}
 If $X\subset\pp^2$
  is a curve of dividing type, then there exists an integer $k$ such that for any $k' \geq k$ there are
  homogeneous polynomials 
  $f,g\in\R[x,y,z]$ of degree $k'$ for which  
  $\mathcal{V}(f,g)\cap X(\mathbb{R})=\emptyset$ 
  and such that $\mathcal{V}(\lambda f+\mu g)$ intersects $X$ only in real points for every $\lambda,\mu\in\R$ not both zero.
\end{Thm}

A curve $X$ of genus $g$ is called an \textit{$M$-curve} if $X(\R)$ has $r=g+1$ connected components. Every $M$-curve is of dividing type.  In the case of $M$-curves we give a complete description of both the separating and hyperbolic semigroups.

\begin{Thm}\label{thm:mcurves}
 Let $X$ be an $M$-curve.
 \begin{enumerate}[a)]
  \item If $g=0$, then $\textnormal{Hyp}(X)=\textnormal{Sep}(X)=\N$.
  \item If $g=1$, then $\textnormal{Sep}(X)=\N^2$ and $\textnormal{Hyp}(X)=\N^2\smallsetminus\{(1,1)\}$.
  \item If $g>1$, then $\textnormal{Sep}(X)=\N^{g+1}$ and $\textnormal{Hyp}(X)=\{d\in{\N^{g+1}}: \, \sum_{i=1}^{g+1} d_i \geq g+3\}$.
 \end{enumerate}
\end{Thm}

Finally, in  Section \ref{sec:locus} we study the hyperbolicity locus of an embedded curve in an example.
Given an embedded  curve $X$ in $\mathbb{P}^n$,  Shamovich and  Vinnikov  asked if  the subset of the Grassmannian $\textnormal{Gr}(n-1, n+1)$ corresponding to the  linear spaces from which the projection of $X$ is separating is connected \cite{Sha14}. In Example \ref{exp:connected}, using  the  hyperbolic semigroup we construct an example where the answer is negative. {In fact, this example is a member of a family of curves constructed independently by  Mikhalkin and Orevkov \cite[Theorem 3]{MikOre}. Their construction immediately implies that there exists an $M$-curve in $\mathbb{P}^3$ of genus $g$ such that its hyperbolicity locus consists of $g+1$ connected components.}

\section{The separating and hyperbolic semigroups}

We begin by showing that the sets $\textnormal{Sep}(X)$ and $\textnormal{Hyp}(X)$ are indeed semigroups.

\begin{Prop}\label{prop:semigroup}
Let $X$ be a curve of dividing type.  Then both $\textnormal{Sep}(X)$ and $\textnormal{Hyp}(X)$ are  semigroups.
\end{Prop}

\begin{proof}
 Let $f_1,f_2\colon X\to\pp^1$ be two separating morphisms.
 Let $X_+$ be one of the connected components of $X(\C)\smallsetminus X(\R)$. Without loss of generality we may assume that $f_1(X_+)=f_2(X_+)=H_+$ is the upper half-plane. 
 Identify $\mathbb{P}^1(\C)$ with $\C \cup \{\infty\}$ and let $\phi \colon \mathbb{P}^1(\C) \to  \mathbb{P}^1(\C) $ be a M\"obius transformation sending the circle $|z|=1$ to $\mathbb{P}^1(\R)$.  Define the map $g \colon X\to\pp^1(\mathbb{C})$  by  $$g(x) = \phi^{-1}(f_1(x))\cdot\phi^{-1}(f_2(x)).$$
The preimage $g^{-1}(z)$ of  any point $z$ with $|z|=1$  is contained in   $\mathbb{P}^1(\R)$. Moreover, this
preimage consists of exactly $d_i(f_1)+d_i(f_2)$ points on $X_i$ for $i = 1, \dots, r$. 
Then the composition $f = \phi\circ g(x)$ is a separating map   which satisfies $d(f)=d(f_1)+d(f_2)$. This proves that $ \textnormal{Sep}(X)$ is a semigroup.

To show that $\textnormal{Hyp}(X)$ is a semigroup, suppose that $\cL_1=f_1^*\cO_{\pp^1}(1)$ and $\cL_2=f_2^*\cO_{\pp^1}(1)$ are both very ample. Then the line bundle $f^*\cO_{\pp^1}(1)=\cL_1\otimes\cL_2$, where $f$ is defined as above, is also very ample. Therefore $\textnormal{Hyp}(X)$ is also a semigroup. 
\end{proof}

When $X = \pp^1$ we have $1 \in \textnormal{Hyp}(\pp^1)$. Therefore,  the next corollary is an immediate consequence of Proposition \ref{prop:semigroup}.

\begin{Kor}\label{cor:p1}
 We have $\textnormal{Sep}(\pp^1)=\textnormal{Hyp}(\pp^1)=\N$. 
\end{Kor}

\begin{Bsp}\label{exp:genus0}
In fact, we have $\textnormal{Sep}(X)=\N$ if any only if $X =\pp^1$, 
{since a map of degree $1$ is an isomorphism.}
\end{Bsp}

\begin{Bem}
 For a dividing curve $X$ of genus $g$, every $d\in \textnormal{Sep}(X)$ with $|d|\geq2g+1$ is also in $\textnormal{Hyp}(X)$. This is because every line bundle $\mathscr{L}$ on $X$ with $\deg(\mathscr{L}) \geq 2g+1$ is very ample \cite[Corollary 3.2 b)]{Hart77}.
\end{Bem}

\begin{Bsp}\label{exp:genus2}
 Let $X$ be a hyperelliptic curve of genus $g=2$ given by $x_2^2=p(x_0,x_1)$  where $p$ is a positive definite form of degree six. We consider $X$ to be a subvariety of  the weighted projective space $\pp^2(1,1,3)$, so that it is non-singular. Then the canonical map is separating. Since $r+g$ must be odd and $f$ is unramified over the real points, we conclude that $r=1$. Thus, we have $2\in \textnormal{Sep}(X)$. By \cite[\S 4.2]{Ahl50} we also have $3\in \textnormal{Sep}(X)$ but we do not have $1\in \textnormal{Sep}(X)$ because the curve is not rational. Thus, we have $\textnormal{Sep}(X)=\N_{\geq2}$.
\end{Bsp}

\begin{Bem}
In general,  the separating and hyperbolic semigroups  do not only depend on $r$ and $g$. For example, it is possible to  construct a hyperelliptic curve of genus three whose canonical map is separating. In that case we have $r=2$. But there are also separating curves of genus three with $r=2$ that are not hyperelliptic and therefore do not admit a separating morphism of degree two.
\end{Bem}

\begin{Bsp}\label{exp:onlyelliptic}
Let $g=1$ and suppose that $r =2$.
Then there is an automorphism of $X$ that sends $X_1$ to $X_2$. Thus, $\textnormal{Sep}(X)$ is stable under the action of the symmetric group $\mathfrak{S}_2$. From {embeddings} to $\pp^2$ we obtain $(2,1),(1,2)\in \textnormal{Hyp}(X)$. One also has $(1,1)\in \textnormal{Sep}(X)\smallsetminus\textnormal{Hyp}(X)$.
\end{Bsp}

\begin{Bsp}\label{exp:plane} 
In this example we show that the separating semigroup of a planar hyperbolic curve $X$ of degree $k$ is not symmetric for  $k \geq 4$. This version of the argument was suggested by Erwan Brugall\'e following our original approach for $k=4$. The number of connected components of $X(\mathbb{R})$ is  $r = \lceil \frac{k}{2} \rceil $.  
 A linear system of rank $2$ on a curve $X$  of genus $g \geq 3$ is unique, \cite[\S 2.3]{namba} or \cite[A.18]{arbarello2013geometry}.  So we can label the connected components 
 $X_1, \dots, X_r$ from the innermost oval $X_1$ to the outermost oval $X_r$ if $k$ is even. If $k$ is odd then $X_{r-1}$ is the outermost and $X_{r}$ is the unique pseudoline.  By our hyperbolicity assumption $(2, 2, \dots, 2) \in \textnormal{Hyp}(X) \subset \textnormal{Sep}(X)$ if $k$ is even and  $(2, 2, \dots, 2,  1) \in \textnormal{Hyp}(X) \subset \textnormal{Sep}(X)$ if $k$ is odd.  
 
 The gonality of a planar curve $X$ of degree $k$ is $k-1$ and moreover every map $f\colon  X \to \mathbb{P}^1$ of degree $k-1$ is induced by a projection $\mathbb{P}^2 \to \mathbb{P}^1$ whose center is a point on $X$,  \cite[\S 2.3]{namba} or \cite[A.18]{arbarello2013geometry}.  
Therefore, we have  $(1, 2, \dots, 2) \in \textnormal{Sep}(X)$ if $k$ is even
and, if $k$ is odd  $(1, 2, \dots, 2, 1) \in \textnormal{Sep}(X)$. However, no other permutation of these degree sequences {is} possible, since a projection whose center is not on the innermost oval of $X$ would not be a   separating morphism.

Therefore, the semigroup $\textnormal{Sep}(X)$ is in general not 
{preserved under the action of the symmetric group.}
Moreover, it follows  that the  semigroup $\textnormal{Sep}(X)$ is an invariant of the curve $X$  which is capable of distinguishing which connected component of $X(\mathbb{R})$ is the innermost oval of a hyperbolic embedding to {$\mathbb{P}^2$} when $g \geq 3$. 
\end{Bsp}

\subsection{Separating morphisms and interlacing sections}
The following definition generalizes the interlacing property for polynomials \cite{fisk06} to sections of line bundles.

\begin{Def}
 Let $\mathscr{L}$ be a line bundle on $X$. Let $s_0,s_1\in\Gamma(X,\mathscr{L})$ be two global sections that both have only simple and real zeros. We say that $s_0$ and $s_1$ \textit{interlace} if each connected component of $X(\R)\smallsetminus\{P \ | \, s_0(P)=0\}$ contains exactly one zero of $s_1$ and vice versa.
\end{Def}

\begin{Lem}\label{lem:interl}
 Let $\mathscr{L}$ be a line bundle on $X$ and let $s_0,s_1\in\Gamma(X,\mathscr{L})$ be global sections that generate $\mathscr{L}$  and that both have only simple and real zeros.    Then the morphism $X\to\pp^1$ given by $x\mapsto (s_0(x):s_1(x))$ is separating if and only if $s_0$ and $s_1$ are interlacing.
\end{Lem}

\begin{proof}
 Assume that $s_0$ and $s_1$ interlace. If the map is not separating, then there is a $\lambda \in \mathbb{R}$ such that $s_0 + \lambda s_1$ has a double zero on $X(\mathbb{R})$. Since $s_0$ and $s_1$ generate $\mathscr{L}$, the section $s_0 + \lambda s_1$ does not vanish on any zero of $s_1$ for any $\lambda$. Thus, because $s_0$ has exactly one zero on each connected component of $X(\mathbb{R}) \smallsetminus \{P \ | \, s_1(P)=0\}$, a double root of $s_0 + \lambda s_1$ is impossible. Therefore, interlacing implies separating.
 
 Conversely, assume that the morphism under consideration is separating. It follows immediately that all zeros of $s_0$ and $s_1$ are real. The other properties of interlacing sections follow from the fact that the restriction of separating morphisms to the real part is unramified \cite[Theorem 2.19]{kummer2015real}.
\end{proof}

To verify the interlacing property of a pair of sections we have the following sufficient criterion which we will use later on.

\begin{Prop}\label{prop:interl}
 Let $\mathscr{L}$ be a line bundle on $X$ and let $s_0,s_1\in\Gamma(X,\mathscr{L})$ be global sections that generate $\mathscr{L}$. Let $s_0$ have only simple  and real zeros. Let $I$ be the set of indices $i$ such that $s_0$ has more than one zero on $X_i$. Assume that for each $i\in I$ there is exactly one zero of $s_1$ on each connected component of $X_i \smallsetminus \{P\ | \ s_0(P)=0\}$. Then $s_0$ and $s_1$ are interlacing and the morphism $X\to\pp^1$ given by $ x\mapsto (s_0(x):s_1(x))$ is separating.
\end{Prop}

\begin{proof}
 As in the proof of the preceding lemma we show that there is no $\lambda\in\mathbb{R}$ such that $s_0+\lambda s_1$ has a double zero on $X(\mathbb{R})$. Indeed, as shown in the proof of the preceding lemma, two zeros on one of the $X_i$ for  $i\in I$ cannot come together as $\lambda$ varies. Since there is only one zero of $s_0$ on the other components, the claim is also true for those.
\end{proof}

\subsection{Conditions for hyperbolic morphisms}

We first 
 point out that a curve being hyperbolic with respect to some linear space is a purely topological property.
 
We begin by recalling linking numbers of spheres embedded in the $n$-dimensional sphere $S^n$. For the general definition of linking numbers and more detailed information we refer to \cite{prasolov}. Suppose that $X$ and $Y$ are disjoint embedded oriented spheres in $S^n$ of dimensions $p$ and $q$ respectively where $n=p+q+1$. Consider the fundamental cycles $[X]$ and $[Y]$ as cycles in the integral homology of $S^n$. There exists a chain $W$ whose boundary is $[X]$. The \textit{linking number} $\lk(X,Y)$ is defined to be  the intersection number of $W$ and $[Y]$.

Now let $K \subset  \pp^n(\R)$ be  
the image of an embedding of $S^1$
and $L \subset \pp^n(\R)$ be a linear subspace of codimension $2$. 
Let $\pi \colon  S^n \to \pp^n(\R)$ be an unramified $2$ to $1$ covering map. Notice that $\pi^{-1}(L)$ is a sphere of dimension $n-2$ in $S^n$ and $\pi^{-1}(K)$ is either an embedded circle or two embedded circles. Define the linking number of $K$ and $L$ in $\pp^n(\R)$ to be the linking number of $\pi^{-1}(K)$ and $\pi^{-1}(L)$ in $S^n$ if $\pi^{-1}(K)$ is a single connected component and define it to be the sum of the linking numbers of $K_1$ with $L$ and $K_2$ with $L$ if $K_1 \cup K_2 = \pi^{-1}(K)$. Now we are able to give a topological characterization of hyperbolic curves in terms of linking numbers.

\begin{figure}
 \includegraphics[width=5cm]{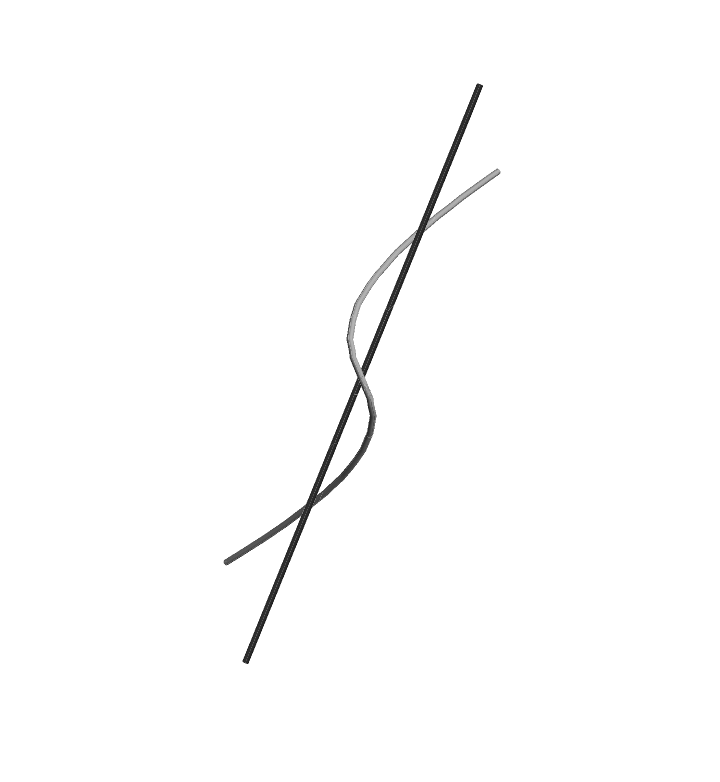} 
\caption{The twisted cubic drawn in green is hyperbolic with respect to the red line.  }
\label{fig:twistedcubic}
\end{figure}

\begin{Prop}\label{prop:topo}
 Let $X\subset\pp^n$ be a 
 curve 
 and $E\subset\pp^n$ be a linear subspace of dimension $n-2$ with $X\cap E=\emptyset$. 
 {Then $X$ is hyperbolic with respect to $E$ if and only if $\deg(X)=\sum_{i=1}^r |\lk(X_i,E(\R))|$.}
{When $X$ is hyperbolic with respect to $E$, then
 the tuple $$(|\lk(X_1,E(\R))|,\ldots,|\lk(X_r,E(\R))|)$$ is the element in $\textnormal{Hyp}(X)$ corresponding to the projection from $E$.}
\end{Prop}

\begin{proof}
 The curve $X$ is hyperbolic with respect to $E$ if and only if every hyperplane $H\subset\pp^n$ that contains $E$ intersects $X$ in $\deg(X)$ many distinct real points. Let $\pi \colon  S^n \to \pp^n(\R)$ be an unramified $2$ to $1$ covering map. For any choice of a hyperplane $H\subset\pp^n$ that contains $E$, the preimage $X=\pi^{-1}(E)$ is a sphere of dimension $n-2$ inside $\pi^{-1}(H)$ which is a sphere of dimension $n-1$. Let $W\subset\pi^{-1}(H)$ be a hemisphere whose boundary is $\pi^{-1}(E)$.
 
  If $X$ is hyperbolic with respect to $E$, then the absolute values of the linking numbers $\lk(X_i,E(\R))$, which are the intersection numbers of the $\pi^{-1}(X_i)$ with $W$,  sum up to $\deg(X)$.
  Conversely, if the intersection number of $W$ with the preimage of $X(\R)$ is $\deg(X)$, then $H$ has (at least) $\deg(X)$ many real intersection points with $X$.
  The final statement about the element of $\textnormal{Hyp}(X)$ arising from the projection from $E$ is immediate. 
\end{proof}

\begin{Bsp}\label{exp:hyperboloid}

Let $Q\subset\pp^3$ be the quadratic surface defined by the equation $x^2+ y^2  =  z^2 + w^2$. Its real part $Q(\R)$ is the hyperboloid. 
For a curve $X$ contained in $Q$, we can describe a topological condition for $X$ to be hyperbolic with respect to the line $E$ given by
$x = y = 0$.

The hyperboloid $Q(\mathbb{R})$ is homeomorphic to the torus $S^1 \times S^1$ {and taking a real line from each of the two rulings of $Q$ gives a pair of generators 
of {$H_1(Q(\mathbb{R})) \cong \Z \oplus \Z$}.} 
We will assume that these lines are oriented in the upwards $z$ direction in the affine chart $w = 1$. 
For each hyperplane $H$ containing the line {$E$} we have $[H \cap Q(\mathbb{R})] = (1, 1) \in H_1(Q(\mathbb{R}))$, 
up to switching the orientation of $H \cap Q(\mathbb{R})$.

If $X \subset Q$, then a connected component of $X(\mathbb{R})$ realizes either the trivial class in $H_1(Q(\mathbb{R}))$ or the class $(p, q)$ for $p, q$ coprime integers. Otherwise, the connected components of $X(\mathbb{R})$ would have non-trivial intersections contradicting the fact that  $X$ is non-singular. 
In order for $X$ to be hyperbolic with respect to {$E$}, no component of $X(\mathbb{R})$ can realize the trivial class. This is because if a connected component $X_i$ is trivial in homology, it is the boundary of a disc contained in $Q$ which would not intersect {$E$} and so $d_i = |\lk(X_i; {E}(\mathbb{R}))| = 0.$ If furthermore $\deg(X) = r\cdot (p+q)$, 
then $X$ must be hyperbolic with respect to {$E$}.

{For example, we can construct a curve $X$ of degree $2k$ in $\mathbb{P}^3$ which is hyperbolic with respect to {$E$} in the following way. Let $X$ be the complete intersection of $Q$ with a hypersurface $S$ which is a small perturbation of the union of hyperplanes $H_i$ with equations of the form $z + a_iw$ for $i = 1, \dots, k$. Then $X(\R)$ consists of $k$ connected components, and each one realizes the class  $(1, -1) \in H_1(Q(\mathbb{R}))$} up to changing the orientations of the connected components of $X(\R)$. By the above remark, the curve $X$ is hyperbolic with respect to {$E$} and thus $(2, \dots, 2) \in \textnormal{Hyp}(X)$.  See the left hand side of Figure \ref{fig:genus3canon}. Let {$E'$} be any real line on $Q$. By our choice of generators of $H_1(Q(\mathbb{R}))$, we can suppose that $[{E'}(\mathbb{R}) ] = (1, 0)$, up to a change in  orientation. Therefore, the set of real points ${E'}(\mathbb{R})$ intersect  $X(\mathbb{R})$ in at least $k$ points. The projection of $X$ from {$E'$} is a separating map of degree $k$ with degree partition $(1, \dots, 1) \in \N^k$. The curve $X$ has genus $(k-1)^2$ and $X(\mathbb{R})$ has $k$ connected components. In particular, if $k \geq3$, then $X$ is not an $M$-curve, see Section \ref{sec:Mcurves}.  
\end{Bsp}

\begin{Bsp}\label{ex:sphere}
We can carry out the construction from the previous example with $Q$ now being the quadratic surface defined by $x^2+y^2+z^2=w^2$. Then $Q(\R)$ is the sphere. 
{Again let  $S$ be the hypersurface which is a small perturbation of the union of hyperplanes $H_i$ with equations of the form $z + a_iw$ with $-1 < a_i < 1$  for $i = 1, \dots, k$.

Let $X$ be the complete intersection of $S$ and $Q$.}  Then the resulting curve $X$ is again hyperbolic with respect to the line {$E$} defined by $x = y = 0$ and we have $(2, \dots, 2) \in \textnormal{Hyp}(X)$.  See the right hand side of Figure \ref{fig:genus3canon}.  As in Example \ref{exp:hyperboloid}, the curve $X$ has genus $(k-1)^2$ and $X(\mathbb{R})$ has $k$ connected components. But we will show that unlike in the preceding example, we have $(1,\ldots,1)\not\in\textnormal{Sep}(X)$ if $k>2$. In fact, there is no real morphism $X\to\pp^1$ of degree $k$.
{Any complex line on $Q$ intersects $X$ in $k$ points and the projection from this line gives a map to $\mathbb{P}^1$ of degree $k$. 
 The gonality of the complexified curve is $k$ and every morphism $X\to\pp^1$ of degree $k$ comes from the projection from a line by \cite{gonal} and \cite{uniqu}.
However, since the surface $Q$ does not contain any real lines, there are no real lines intersecting $X(\R)$ in more than two points proving the claim that there are no real morphisms of degree $k$ to $\mathbb{P}^1$. }
  \end{Bsp}
  
  \begin{figure}[ht]
 \includegraphics[width=4cm]{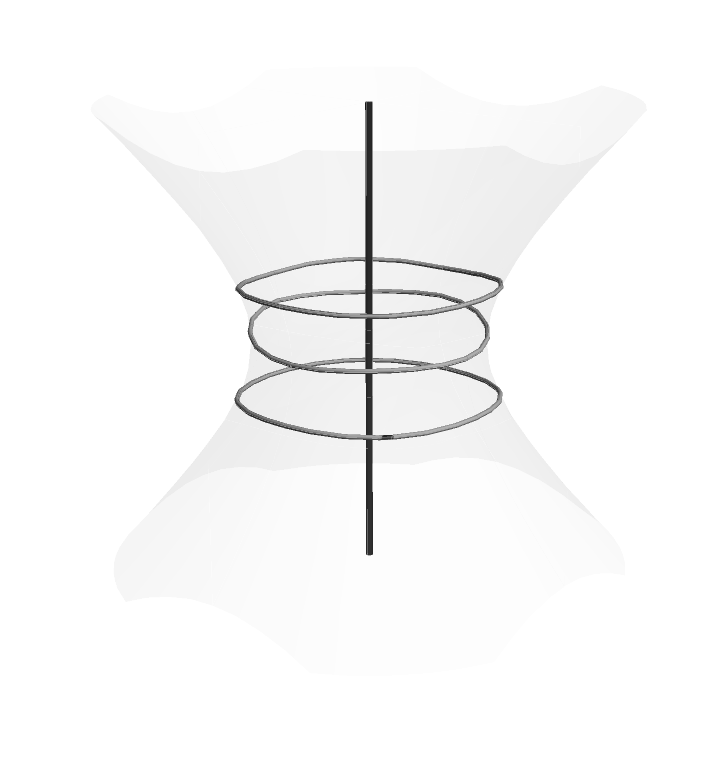} \quad
 \includegraphics[width=4cm]{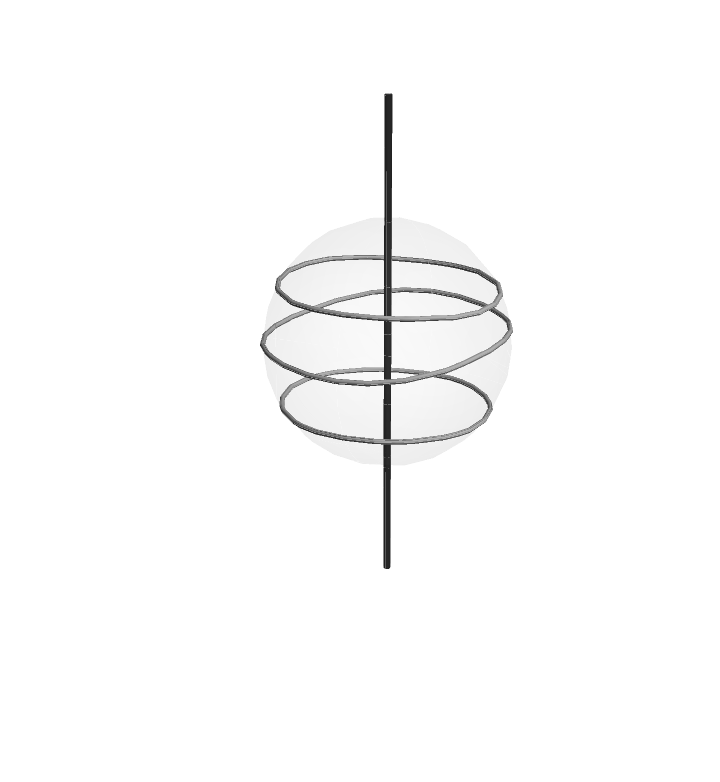}
\caption{The two canonical curves of genus four from Examples \ref{exp:hyperboloid} (left) and \ref{ex:sphere} (right)  with real part having three connected components that are hyperbolic with respect to the red line. The curve on the left has $(1,1,1)$ in its separating semigroup whereas the one on the right does not. } 
\label{fig:genus3canon}
\end{figure}

\begin{Bsp}
 
Let $X$ be a curve of genus $g$ with $X(\R)$ having $r$ connected components. 
Assume that there is a separating morphism $f \colon X\to\pp^1$ with the property that $f^*\cO_{\pp^1}(1)$ is the canonical line bundle. The degree of a divisor of a non-zero {real} holomorphic differential form on $X$ restricted to any connected component of $X(\mathbb{R})$ must be even, see for example \cite[Proposition 4.2]{grossharris}. Therefore, we have that $d(f)=2d'$ for some $d'\in\N^r$. This implies that $r\leq g-1$. Examples \ref{exp:genus2}, \ref{exp:plane}, and \ref{exp:hyperboloid} show that for  $g=2,3,4$ such a morphism exists for a curve with $r=g-1$. 
For planar curves $X\subset\pp^2$ of degree $d\geq 4$ the canonical bundle is given by $\mathcal{O}_X(d-3)$. Furthermore, it was shown in \cite{touze2013totally} that for some planar sextic curves with $9$ ovals there exists a pencil of cubics without base points on $X$ that gives rise to a separating morphism. Thus for $g=10$ and $r=9$ we can also find such a morphism. For curves of genus different from $2,3,4$ and $10$ we do not know if this is the case.   
\end{Bsp}

{The next two lemmas exclude specific elements from the hyperbolic semigroup. } 

\begin{Lem}\label{lem:rat111}
 {Let  $X$ be a curve of dividing type. If 
  $d = (1, \dots, 1)\in \textnormal{Hyp}(X)$, then $X$ is a rational curve, and hence $d = (1)$.
 }  
\end{Lem}

\begin{proof}
 Assume that there is a hyperbolic embedding of $X$ to $\pp^3$ with each component having degree $1$. The linear projection of $X$ to the plane from any point not on $X$ will send each $X_i$ to a pseudoline. In $\pp^2$ each two pseudolines intersect, which implies that the image of $X$ will have at least $\frac{1}{2}r(r-1)$ simple nodes given that the center of projection was chosen generally enough. On the other hand, the degree is $r$. By the genus--degree formula this implies that \[g\leq \frac{1}{2}(r-1)(r-2)-\frac{1}{2}r(r-1)=1-r.
 \]
{Therefore, if $r = 1$ 
then  $X$ is a rational curve and  if  $r>1$ then $X$ cannot be irreducible. This proves the lemma. }
\end{proof}

\begin{Lem}\label{lem:not211}
{Let $X$ be a curve of dividing type such that  $r>2$. Then no permutation of $(2, 1, \dots, 1)$ is in $\textnormal{Hyp}(X)$.}
\end{Lem}

\begin{proof}
 Assume that $(2, 1, \dots, 1)\in\textnormal{Hyp}(X)$ and let $X\subset\pp^3$ be a realizing embedding. Let $X$ be hyperbolic with respect to a line $L$ and let $H\subset\pp^3$ be any hyperplane containing $L$. Let $H\cap X_1=\{P_1,P_2\}$ and $H\cap X_i=\{Q_i\}$ for all $i=2,\ldots,r$. Note that $X$ is not contained in any plane since $r>2$. Thus, we can assume that there is a $Q_{i_0}$ that is not on the line spanned by $P_1$ and $P_2$. The set $A$ of all real lines in $H$ that contain neither $P_1$ nor $P_2$ has two connected components. Let $L' \subset H$ be a line through $Q_{i_0}$ that is in the same connected component of $A$ as $L$. Every hyperplane containing $L'$ intersects $X_1$ in at least two real points because of Proposition \ref{prop:topo} and $|\lk(X_1,L(\R))|=|\lk(X_1,L'(\R)|$. Also it intersects every $X_i$ with $i\geq2$ in at least one real point. Now any hyperplane $H'$ spanned by $L'$ and a point $Q'\neq Q_{i_0}$ on $X_{i_0}$ would intersect $X$ in more than degree many points.
\end{proof}

\section{Some results for the general case}

In general it is not easy to determine the separating and the hyperbolic semigroup for a given curve. In this section we provide a method that allows us under some reasonable assumptions to construct from a separating morphism another separating morphism of one degree higher. The main result of this section is that any embedding of a separating curve of high enough degree is hyperbolic. Therefore, hyperbolic embeddings are the rule rather than the exception.

\begin{Lem}\label{lem:existsec}
 Let $\mathscr{L}$ be a line bundle on $X$. Let $s_1,s_2\in\Gamma(X,\mathscr{L})$ be two global sections. Let $(s_1)_0=P_0+\ldots+P_n$ with pairwise distinct $P_j\in X(\R)$ and $s_2(P_0)\neq0$. Let $U_j\subset X(\R)$ be an open neighbourhood of $P_j$ for all $j=1,\ldots,n$. There is an open neighbourhood $U_0\subset X(\R)$ of $P_0$ such that for every $Q_0\in U_0$ there are $Q_j\in U_j$ for $j=1,\ldots,n$ and $s_3\in\Gamma(X,\mathscr{L})$ with $(s_3)_0=Q_0+\ldots+Q_n$.
\end{Lem}

\begin{proof}
 It suffices to show the claim in the case where the $U_j$ are pairwise disjoint and do not contain $P_0$ for $j=1,\ldots,n$. Let $f$ be a rational function on $X$ with the property that $(f)=(s_1)_0-(s_2)_0$. Since $f$ has only  simple and real zeros, there is a $c>0$ such that for all $\epsilon\in ]-c,c[$ the rational function $f+\epsilon$ has a zero in each $U_j$ for all $j$ with $P_j$ a zero of $f$. Now let $U_0\subset X(\R)$ be an open neighbourhood of $P_0$ which is disjoint from each of the $U_j$ and is contained in $f^{-1}(]-c,c[)$. Then for every $Q_0\in U_0$ there is an $\epsilon\in ]-c,c[$ such that $Q_0$ is a zero of $f+\epsilon$. Note that $f+\epsilon$ has the same poles of the same orders as $f$. Thus, the effective divisor $(f+\epsilon)+(s_2)_0$ is of the form $Q_0+\ldots+Q_n$ where $Q_j\in U_j$.
\end{proof}

\begin{Prop}\label{prop:semigroupgen}
 Let $\mathscr{L}$ be a line bundle on $X$. Let $s_0,s_1\in\Gamma(X,\mathscr{L})$ be two global sections that interlace. Let $D=(s_0)_0$ and let $P\in X(\R)$ with $s_0(P)\neq0$ such that $\ell(D+P)>\ell(D)$. Then there are interlacing sections $s_0', s_1' \in\Gamma(X,\mathscr{L}')$ such that $(s_0')_0=D+P$ where $\mathscr{L}'$ is the line bundle corresponding to $D+P$.
\end{Prop}

\begin{proof}
 Without loss of generality, we can assume that $s_1(P)\neq0$. Let $(s_1)_0=D'=P_1+\dots+P_n$. Since $\ell(D+P)>\ell(D)$, there is a rational function $g\in\mathscr{L}(D+P)$ that has a pole at $P$. Then $D''=D+P+(g)$ is an effective divisor with $P\not\in\Supp D''$. The effective divisors $D+P$, $D'+P$, and $D''$ correspond to global sections $s_0'$, $f_1$, and $f_2$ of $\mathscr{L}'$, respectively. Let $U_j$ be the connected component of $X(\R)\smallsetminus\Supp(D+P)$ that contains $P_j$ for all $j=1,\dots,n$.  Applying Lemma \ref{lem:existsec} to $f_1$ and $f_2$ shows that there is a global section $s_1'$ of $\mathscr{L}'$ such that $s_0'$ and $s_1'$ interlace.
\end{proof}

\begin{Bem}\label{rem:special}
 The condition $\ell(D+P)>\ell(D)$ is for example satisfied when $D$ is non-special.
\end{Bem}

\begin{Kor}\label{cor:orthants}
Let $X$ be a curve of dividing type. Let  $d\in\textnormal{Sep}(X)$ and $l$ be the number of indices where $d_i$ is odd. If $|d|+l\geq 2g-1$,  then $d+\Z_{\geq0}^r\subset\textnormal{Sep}(X)$. 

\end{Kor}

\begin{proof}
By \cite[Theorem 2.5]{nsd} the divisor corresponding to a separating morphism realizing $d$ is non-special. By Remark \ref{rem:special} the statement then follows from  Proposition \ref{prop:semigroupgen}.
\end{proof}

\begin{proof}[Proof of Theorem \ref{thm:alwayshyperbolic}]
 By Ahlfors' theorem \cite[\S 4.2]{Ahl50} there is a separating morphism $f \colon X\to\pp^1$. By Proposition \ref{prop:semigroup} we can furthermore assume that the degree of $f$ is more than $2g-2$. Let $\mathscr{L}=f^*\mathcal{O}_{\pp^1}(1)$ and $s_0,s_1\in\Gamma(X,\mathscr{L})$ be two global sections that interlace. Then $D_0=(s_0)_0$ is non-special. By \cite[Cor. 2.10, Rem. 2.14]{scheider}, there is an integer $n>0$ such that every real divisor on $X$ of degree at least $n$ is linearly equivalent to a sum of distinct points from $X(\R)\smallsetminus\Supp D_0$. We will show that the claim holds for $k=\max(2g+1,n+\deg D_0)$. Indeed, let $D$ be a divisor with $\deg D\geq k$. Then $D-D_0$ is linearly equivalent to a  sum of distinct points from $X(\R)\smallsetminus\Supp D_0$ since $\deg(D-D_0) \geq n$. An iterated application of the previous proposition shows that the corresponding embedding of $X$ to $\pp(\mathscr{L}(D)^\vee)$ is hyperbolic.
\end{proof}

\begin{proof}[Proof of Theorem \ref{thm:pencils}]
 The existence of a totally real pencil of curves of degree $k$  satisfying our assumptions  follows immediately by applying Theorem \ref{thm:alwayshyperbolic} to the line bundle $\mathcal{O}_{X}(k)$.
\end{proof}

\begin{Bem}
 If we allow base points on the curve, then the existence {of a totally real pencil} 
 simply follows from Ahlfors' theorem \cite[\S 4.2]{Ahl50}. Indeed, let $X\subset\pp^2$ be a curve of dividing type and $f \colon X\to\pp^1$ be a separating morphism defined by two sections $s_0,s_1$ of a suitable line bundle. The rational function $\frac{s_0}{s_1}$ can be expressed in the coordinates $x,y,z$ of $\pp^2$ as the fraction of two homogeneous forms $f,g$ in $x,y,z$ of the same degree. Then $\mathcal{V}(\lambda f+\mu g)$ intersects $X$ only in real points for every $\lambda,\mu\in\R$ not both zero. {However, the intersection $\mathcal{V}(f,g)\cap X$ is in general not empty, so the statement of  Theorem \ref{thm:pencils} does not follow from this argument.}
\end{Bem}

\begin{frag}\label{frag:k}
{What are the smallest possible values for $k$ in Theorem \ref{thm:alwayshyperbolic} and Theorem \ref{thm:pencils}?}
\end{frag}

\begin{Bem} 
{Since our proof of Theorem \ref{thm:pencils} is not constructive,
it further motivates the desire for a bound on $k$ in   Question \ref{frag:k}.
In \cite{touze2013totally}, it  was shown that for plane sextic curves of type $\langle 2 \amalg 1 \langle 6 \rangle\rangle$ and $\langle 6 \amalg 1 \langle 2 \rangle\rangle$ one can choose $k=3$ in Theorem \ref{thm:pencils}. 
}
\end{Bem}

\begin{Bsp}
We consider a planar hyperbolic quartic, for example the \textit{Vinnikov quartic} $X\subset\mathbb{P}^2$ {from \cite[Example 4.1]{quarticexp}} which is the planar curve defined by $$2 x^4 + y^4 + z^4 - 3 x^2 y^2 - 3 x^2 z^2 + y^2 z^2=0.$$ We have $g=3$ and $r=2$. Moreover, the curve $X$ is of dividing type since it is hyperbolic with respect to $(0:0:1)$. Let $X_1$ be the inner oval and $X_2$ be the outer oval. We have $(1,2), (2,2)\in\textnormal{Sep}(X)$, but  Example  \ref{exp:plane} shows that  $(1,1), (2,1)\not\in\textnormal{Sep}(X)$.  We can also realize $(3,2)$ and $(1,3)$ with pencils of conics having $3$ and $4$ base points {on $X(\R)$} respectively, see Figure \ref{fig:quartic}. In coordinates, the separating morphisms are given by the rational functions $\frac{(2z + \sqrt{2} x - 2 y)(-2z + \sqrt{2} x - 4 y)}{xy}$ and $\frac{(2z + \sqrt{2} x - 2 y) (z + x + y)}{xy}$, respectively. 
 By Corollary \ref{cor:orthants}, we have that $(1,2)+\mathbb{Z}_{\geq0} \subset\textnormal{Sep}(X)$. 
 This determines the separating semigroup of the Vinnikov curve except for some cases when $d_2 = 1$.  
In fact, we do not know whether or not  $(n,1)$ is in $\textnormal{Sep}(X)$ for $n \in \N \smallsetminus \{1, 2\}$.
\end{Bsp}

\begin{figure}[ht]
 \includegraphics[width=4cm]{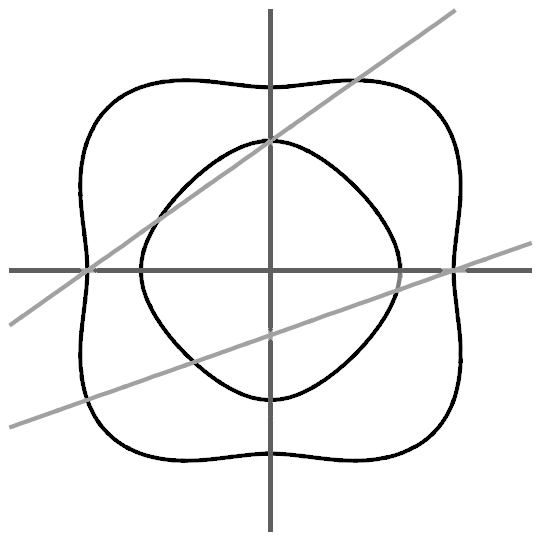} \quad
 \includegraphics[width=4cm]{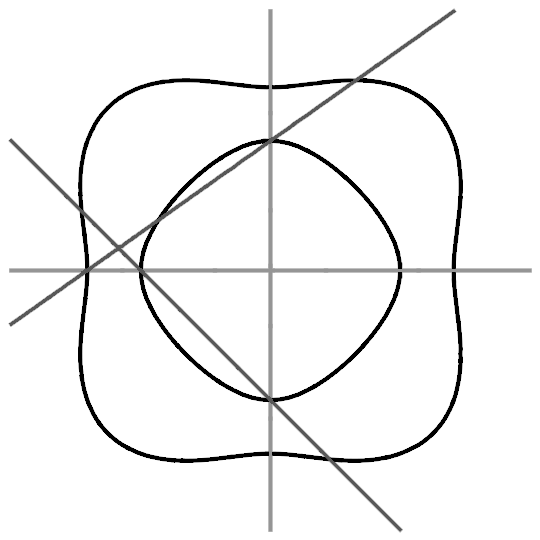}
\caption{Realizing $(3,2)$ and $(1,3)$ with pencils of conics having $3$ and $4$ base points on {$X(\R)$}.  }
\label{fig:quartic}
\end{figure}

\section{M-curves}\label{sec:Mcurves}

Recall that an $M$-curve $X$ has exactly $g+1$ connected components in $X(\R)$ and that every $M$-curve is of dividing type. Here we prove Theorem \ref{thm:mcurves}, which  completely determines the separating and hyperbolic semigroups of $M$-curves.

\begin{proof}[Proof of Theorem \ref{thm:mcurves}]
Part $a)$ is just Corollary \ref{cor:p1}.
By \cite[Proposition 4.1]{gabard} or \cite[\S 4.2]{Ahl50} the all ones vector is in $\textnormal{Sep}(X)$ for any genus. Moreover, any choice of divisor $D=P_1+\ldots+P_{g+1}$ where $P_i\in X_i$ for $i=1,\ldots,g+1$ is non-special by \cite[Theorem 2.4]{huisman}. Thus $\ell(D)=2$ and any non-constant rational function $f\in\mathscr{L}(D)$ gives rise to a separating morphism realizing $(1,\ldots,1)\in\textnormal{Sep}(X)$. By iterated application of Proposition \ref{prop:semigroupgen} and Remark \ref{rem:special} we find that  $\textnormal{Sep}(X) = \mathbb{N}^r$. This proves the claim on $\textnormal{Sep}(X)$ for $X$ of any genus.

By Halphen's theorem \cite[Proof of Proposition 6.1] {Hart77} the space of divisors which are not very ample is of codimension at least one in the space of all divisors of degree at least $g+3$. Therefore given an effective divisor $D$ realizing $d=(d_1, \dots, d_r) \in \textnormal{Sep}(X)$ with $|d| \geq g+3$ obtained from the above construction, we can assume that it is very ample. Thus $d=(d_1, \dots, d_r) \in \textnormal{Hyp}(X)$ whenever $|d| \geq g+3$. For the hyperbolic semigroup in part $b)$ note that $(1,1)\not\in\textnormal{Hyp}(X)$ by Lemma \ref{lem:rat111} and that $(1,2),(2,1)\in\textnormal{Hyp}(X)$ by Example \ref{exp:onlyelliptic}. Part $c)$ follows from the above and Lemmas \ref{lem:rat111} and \ref{lem:not211}.

\end{proof}

\section{Domain of hyperbolicity}\label{sec:locus}

Given a  curve $X\subset\mathbb{P}^n$ {not contained in a hyperplane},  we study the hyperbolicity locus $\mathcal{H}(X)$ of $X$. 
Define $$\mathcal{H}(X) := \{ L \ | \ X \text{ is hyperbolic  with respect to } L \} \subset\text{Gr}(n-1, n+1),$$ 
where $\text{Gr}(n-1, n+1)$ is the space of   codimension $2$ subspaces of $\mathbb{P}^n$. 
In \cite[Question 3.13]{Sha14} the authors asked  whether $\mathcal{H}(X)$ is connected. This question is motivated by the case of planar curves where the answer is yes \cite[Theorem 5.2]{HV07}. Further evidence towards a positive answer was given by \cite[Theorem 7.2]{Sha14} which states that $\mathcal{H}(X)$ is the intersection of $\text{Gr}(n-1, n+1)$ with a convex cone in $\pp^{N-1}(\R)$ where $N=\binom{n}{n-2}$. Here we consider $\text{Gr}(n-1, n+1)$ as a subset of $\pp^{N-1}$ via the Pl\"ucker embedding.
However, the next example shows that the hyperbolicity locus is in general not connected. 

\begin{Bsp}\label{exp:connected}
Let $X\subset\pp^2$ be a planar elliptic curve with $X(\R)$ having two connected components $X_1$ and $X_2$. We can find quadrics $q_1, q_1', q_2, q_2' \in\cO_X(2)$ such that $q_i$ and $q_i'$ interlace. 
 Furthermore, we can assume that  $q_i$ intersects $X_i$ in exactly $4$ points. 
 
Consider  the image of $X$ in $\pp^5$ under the second Veronese embedding. The quadrics $q_i$ and $q_i'$  determine hyperplanes $H_i$ and $H_i'$ in $\pp^5$. Because $q_i$ and $q_i'$ interlace, the image of $X$ is hyperbolic with respect to the $3$-planes $L_i=H_i\cap H_i'$. The projection maps produce the elements $(2, 4)$ and $(4, 2)$ of the semigroup $\textnormal{Hyp}(X)$.  Therefore, the $3$-planes $L_1$ and $L_2$ can not lie in the same connected component of $\mathcal{H}(X)$.
\end{Bsp}

\begin{figure}[ht]
 \includegraphics[width=5cm]{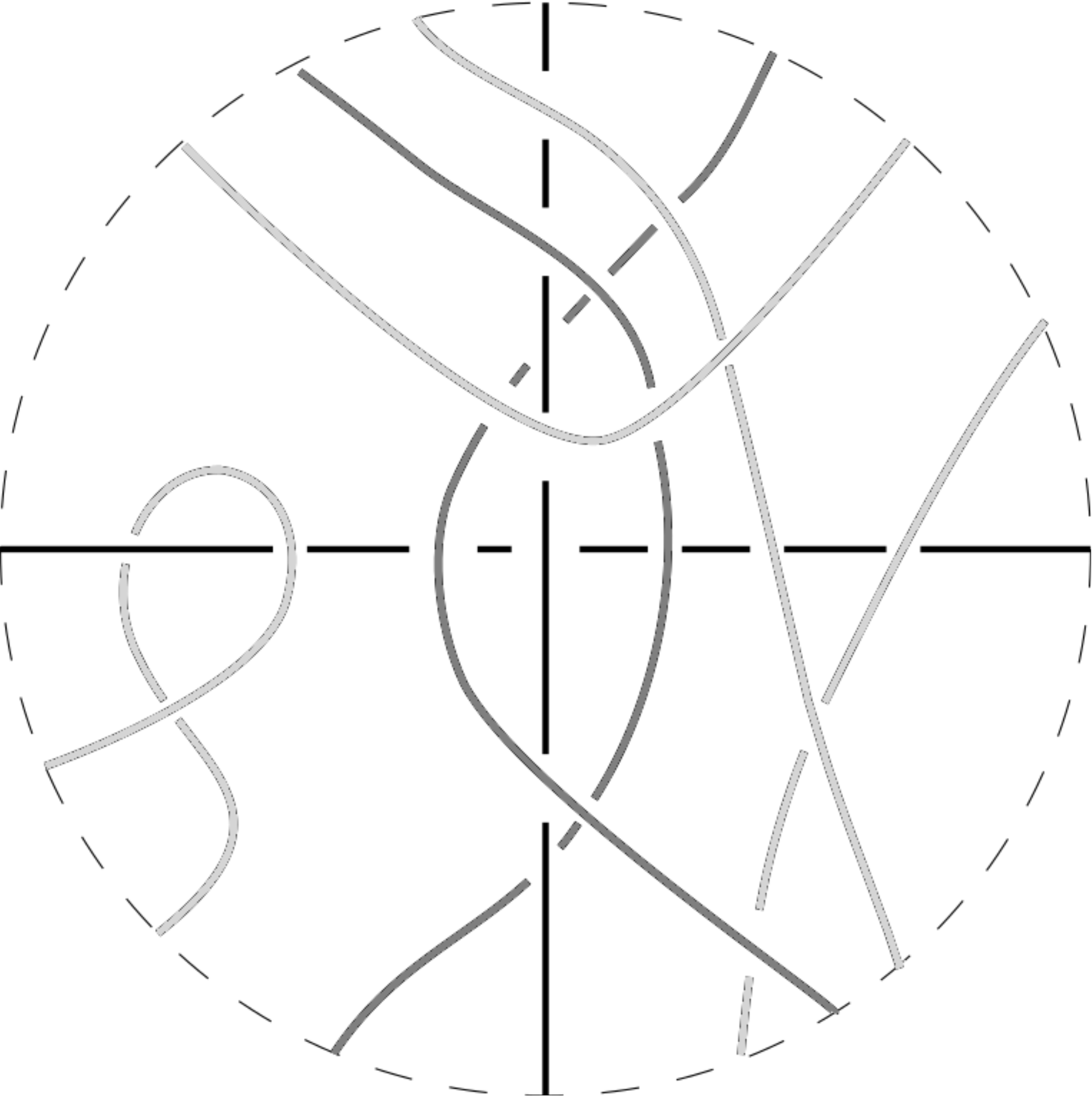} 
\caption{The link diagram of the curve from Example \ref{exp:explicit} together with two lines of hyperbolicity giving rise to different degree partitions.}
\label{fig:notconnected}
\end{figure}

\begin{Bsp}\label{exp:explicit}
One can even obtain an example of a hyperbolic curve $X\subset\pp^3$ with $\mathcal{H}(X)$ not connected. For this we proceed as is Example \ref{exp:connected} to construct two totally real pencils of conics,  $\lambda q_0 + \mu q_1$ and $\lambda p_0 +  \mu p_1$, that give rise to different degree partitions on a plane elliptic curve. Then $X$ is obtained as the image the plane elliptic curve under the map $$x\mapsto (q_0(x):q_1(x):p_0(x):p_1(x)).$$ 
For a precise example, take the  image of the plane curve defined by $-z^3 + 2 x z^2 - x^3 + y^2 z=0$ under the above map with $q_0=x y$, $q_1=x^2-y^2$, $p_0=y(x - 4z)$ and $p_1=(3x - 4z - y)(2x + y)$.

Figure \ref{fig:notconnected} shows the link diagram, up to isotopy, following a linear projection of this embedding back to $\mathbb{P}^2$. The two components of the links are depicted in red and blue, and the image of the two lines of hyperbolicity are shown in black. The real projective plane is depicted as a disk with antipodal boundary points identified. 
\end{Bsp}

\begin{frag}
 Is there an example of a non-singular  curve $X\subset\mathbb{P}^n$ where two connected components of $\mathcal{H}(X)$ give rise to the same element in $\textnormal{Hyp}(X)$?
\end{frag}

\bigskip

 \noindent \textbf{Acknowledgements.}
We would like to thank Erwan Brugall\'e, Ilia Itenberg and Claus Scheiderer for useful discussions. {We are very grateful to Erwan Brugall\'e, Lionel Lang, Bernd Sturmfels, and to an anonymous referee for helpful comments on a preliminary version of this manuscript. }

\bibliographystyle{amsalpha}
\bibliography{biblio}
 \end{document}